\documentclass[10pt,reqno,titlepage]{amsart}
\usepackage{graphicx}
\usepackage{footmisc}
\usepackage{indentfirst,csquotes}
\usepackage{amssymb,amsthm,amsmath,mathrsfs,scalerel,stackengine}
\usepackage{xcolor, hyperref, fancyhdr, etoolbox, paralist}
\usepackage{lipsum}
\theoremstyle{definition}
\newtheorem{teo}{Theorem}
\newtheorem{defn}[teo]{Definition}

\newtheorem{cor}[teo]{Corollary}
\newtheorem{rem}[teo]{Remark}

\newcounter{mybibstartvalue}
\newcommand{\vd}{\mathrm{d}}

\makeatletter
\newcommand{\newsectionstyle}{%
  \renewcommand{\@secnumfont}{\bfseries}
  \renewcommand\section{\@startsection{section}{2}%
    \z@{.5\linespacing\@plus.7\linespacing}{-.5em}%
    {\normalfont\bfseries}}%
}
\let\oldsection\section
\let\old@secnumfont\@secnumfont
\newcommand{\originalsectionstyle}{%
  \let\@secnumfont\old@secnumfont
  \let\section\oldsection
}
\makeatother
\hypersetup{ colorlinks=true, linkcolor=black, filecolor=black, urlcolor=black }

\usepackage{lipsum}
\def\quotient#1#2{%
    \raise1ex\hbox{$#1$}\Big/\lower1ex\hbox{$#2$}%
}
\stackMath
\newcommand\reallywidecheck[1]{%
\savestack{\tmpbox}{\stretchto{%
  \scaleto{%
    \scalerel*[\widthof{\ensuremath{#1}}]{\kern-.6pt\bigwedge\kern-.6pt}%
    {\rule[-\textheight/2]{1ex}{\textheight}}
  }{\textheight}%
}{0.5ex}}%
\stackon[1pt]{#1}{\scalebox{-1}{\tmpbox}}%
}
\makeatletter
\patchcmd{\@maketitle}
{\if@titlepage \newpage \else}
{\if@titlepage
 \vspace{\baselineskip}
 \else}
{}{}
\makeatother
\title{Distributional Solution and Spectral Shift Function of Heun Differential Equation} 
\author{Ubong Sam IDIONG}
\begin{document}
    \begin{abstract}
In this work, the Heun operator is written as an element in the universal enveloping algebra of the Lie algebra $\mathscr{G}=\mathscr{L}(G)$ of the Lie group $G=SL(2,\mathbb{C})$. The Green function and the spectral shift function of the exactly solvable Heun operator from the resulting Lie algebraic equation are obtained via Fourier transform over $G$.
\end{abstract} 
\maketitle
\let\thefootnote\relax

\date{\today}

\section{Introduction}
Heun differential equation (HDE) is a second-order linear ordinary differential equation in the complex domain defined by
\begin{equation}\label{qhd}
  z(z-1)(z-a)\Psi''(z)+[\gamma(z-1)(z-a)+\delta z(z-a)+ \varepsilon z(z-1)]\Psi'(z)+(\alpha\beta z-q)\Psi(z)=0,
\end{equation}
such that $\alpha,\beta,\gamma,\delta,\varepsilon$ satisfy the constraint equation
\begin{equation}\label{cstr}
  \alpha+\beta+1=\gamma+\delta+\varepsilon.
\end{equation}
 (see ~\cite{RA1}, \S(4.2.3), p. 47). Here, $a\in\mathbb{C}\setminus \{0,1\}$ and $q\in\mathbb{C}$  is the accessory parameter, which plays the role of spectral parameter in most applications.  This equation has regular singularities $0,1, a,\infty$ with corresponding exponents $\{0, 1-\gamma\},\{1,1-\delta\}, \{a, 1-\varepsilon\}$  and $\{\infty, \alpha\beta\}$ respectively.
The search for solutions to HDE has attracted much research interest. The HDE, using the formula, $2^{n-1}n!,$ is known to have 192 solutions, which corresponds with the number of regular singularities $n=4$  (\cite{MRS}:~811).

The concept of distributional solutions of Fuchsian equations has been studied in  \cite{LJK, ERK1, ERK2}. These studies have been limited to studying ordinary differential equations with three regular singularities $0,1$ and $\infty$ using Laplace transform. This paper extends the study to operators with four regular singularities, $0,1, a,\infty$ of which the Heun equation is typical.

A grasp of Fuchsian differential operators' spectral shift functions (SSFs) is necessary to comprehend their spectrum properties, particularly how they vary in response to disturbances on the Reimann surfaces. In quantum mechanics, mathematical physics, and inverse spectrum problems, it is essential to extend the SSFs which are being studied on the real line to the projective complex line. A thorough analysis of these functions improves the mathematical context and produces robust tools. It also illustrates the stability properties of solutions to differential equations under small disturbances. Research in this area may lead to novel techniques and mathematical theorems. 

Furthermore, the literature on applications of SSFs to Fuchsian-type equations is scarce.  This is because SSFs are difficult to study on the Riemann sphere. In studying SSFs of Fuchsian equations, the computation of Green functions is inevitable. Our motivation and main objective in this paper include computing the Green kernel of the exactly solvable HDE, and by way of application, examine the compactness of its associated Green integral relation and its SSF.

In its canonical form, the Heun operator given by
 $$H_c=\frac{\vd^2}{\vd z^2}+\left(\frac{\gamma}{z}+\frac{\delta}{z-1}+\frac{\varepsilon}{z-a}\right)\frac{\vd}{\vd z}+\frac{\alpha\beta z-q}{z(z-1)(z-a)}.$$

 All other Fuchsian equations having four regular singularities in the extended complex plane $\mathbb{CP}^{1}=\mathbb{C}\cup\{\infty\}$ can be transformed into \eqref{qhd} (\cite{OFW}, \S 31.2:~711). For our interest, equation~\eqref{qhd} can be re-written in the form
\begin{equation}\label{qhd1}
    (z^{3}-(1+a)z^{2}+az)\Psi''(z)+[(\gamma+\delta+ \varepsilon)z^{2}-((1+a)\gamma+\delta+\varepsilon)z+\gamma a]\Psi'(z)
 +(\alpha\beta z-q)\Psi(z)=0.
\end{equation}

The outline of follows:  Section~\ref{prel} shall consist of mathematical preliminaries leading to Lie algebraization of the Heun differential equation and eigenvalue of an exactly solvable Heun differential equation. In Section \ref{GrHu}, we present the main results in this paper, which include: the distributional solution of the HDE is also evaluated; Green function of the equation; the SSF of Heun operator; and finally, and compactness integral relation is also determined. In the last Section, vital conclusions are drawn
 on the work.
\section{Prelimnaries}\label{prel}
In this section, a preliminary result which transforms the Heun operator as an element in the universal enveloping algebra $U(\mathscr{G})$ of the Lie algebra $\mathscr{G}=\mathscr{L}(G)$ of the Lie Group $G=SL(2,\mathbb{C})$ is given without proof.  $U(\mathscr{G})$ is also a Lie algebra and can be a finite-dimensional algebra. Elements of $U(\mathscr{G})$ are called invariant differential operators acting on the Lie group $G.$ Standard reference materials on UEA include ~\cite{DJ} and (\cite{KJA}, \S4.8:85-86) among others.
The Lie algebra under consideration here is the three-dimensional Lie algebra $sl(2,\mathbb{C})$ (cf: \cite{MIL}, \S 1, Eq.(1.61), p. 20) which is spanned by first-order differential operators in complex variable $z$:
 \begin{equation}\label{Lgen}
J_{+}:=z^{2}\frac{\vd}{\vd z}-2jz,\;\;\; J_{0}:=z\frac{\vd}{\vd z}-j,\;\;\; J_{-}:=\frac{\vd}{\vd z}
 \end{equation}
where $j=\frac{n}{2}, n\in\mathbb{Z}.$
Thus, the operators \eqref{Lgen} satisfy the commutation relations,
\begin{equation}\label{qalg}
[J_{+},J_{-}]=2J_{0},\;[J_{0},J_{+}]=J_{+} , \; [J_{0},J_{-}]=-J_{-}.
\end{equation}
which is isomorphic to the matrix generators of the Lie algebra $sl(2,\mathbb{C}).$

In what follows, we present some preliminary results on the transformations of Heun differential operator to an element of the UEA of $SL(2,\mathbb{C})$ and its gauge transformation.
\begin{teo}  The Heun differential expression in ~\eqref{qhd1} as an element of universal enveloping algebra of $SL(2,\mathbb{C})$
can be written in terms of the generators of the Lie algebra $sl(2,\mathbb{C})$ as
\begin{eqnarray*}
-H\psi &:=&\bigg\lbrace\frac{1}{2}[ J_{+}J_{0}+J_{0}J_{+}]-\frac{(1+a)}{2}[J_{+}J_{-}+J_{-}J_{+}]+\frac{a}{2}[ J_{0}J_{-}+J_{-}J_{0}]\\
 &&+[\gamma+\delta+\varepsilon
 +\frac{3}{2}(2j-1)]J_{+}+[(2j-1-\gamma)(1+a)-\delta-\varepsilon]J_{0}\\
 &&+a[\gamma-\frac{(2j-1)}{2}]J_{-} +j[(2(1-j)+\gamma)(1+a)+\delta+\varepsilon]-q\bigg\rbrace\psi .
 \end{eqnarray*}
 provided that $8j^{2}+2j(\alpha+\beta-1)+\alpha\beta=0$ and its explicit form depending on the spin number $j$ becomes
$$-H_{j}=z(z-1)(z-a)D^{2}+[\rho_{j} z^{2} +\sigma_{j}  z+\tau_{j}]D+\alpha_{j}\beta_{j} z, \;\; (D=\vd/\vd z)$$
with
\begin{eqnarray*}
 \rho_{j} &=& \gamma +\delta +\varepsilon=\alpha+\beta+1;\\
 \sigma_{j} &=& [2(2j-1)-\gamma] (a+1) -\delta - \varepsilon;\\
\tau_{j} &=& a(\gamma -2j+1));\\
\alpha_{j}\beta_{j} &=&-2j(2j+\alpha+\beta);\;\;\mathrm{and} \\
q_{j}&=& j[\big(2(1-j)+\gamma\big)(1+a)-\delta-\varepsilon].
\end{eqnarray*}
\end{teo}
The condition for the Heun operator $H$ to be exactly solvable is that the coefficient of the term $J_{+}$ in positive grading must be equal to zero.
\begin{cor} Exactly solvable Heun Hamiltonian is given by
\begin{eqnarray*}
 - H_{e}&=&\frac{1}{2}[ J_{+}J_{0}+J_{0}J_{+}]-\frac{(1+a)}{2}[J_{+}J_{-}+J_{-}J_{+}]+\frac{a}{2}[ J_{0}J_{-}+J_{-}J_{0}]\\
 &&+[(2j-1-\gamma)(1+a)-\delta-\varepsilon]J_{0}+a[\gamma-\frac{(2j-1)}{2}]J_{-}\\
 &&+j[(2(1-j)+\gamma)(1+a)+\delta+\varepsilon]-q.
\end{eqnarray*}
and setting $j=\frac{n}{2}$, in the expanded form in terms the product of generators $J_{+},J_{0},J_{-}$ in \eqref{Lgen} is given by
\begin{multline*}
-H_{\frac{n}{2},e}=z(z-1)(z-a)D^{2}+[\rho_{\frac{n}{2}} z^{2} +\sigma_{\frac{n}{2}}  z+\tau_{\frac{n}{2}}]D+\alpha_{\frac{n}{2}}\beta_{\frac{n}{2}} z, \;\; (D=\vd/\vd z)
\end{multline*}
with
\begin{eqnarray*}
 \rho_{\frac{n}{2}} &=& \frac{3(1-n)}{2};\\
 \sigma_{\frac{n}{2}} &=&( n-1-\gamma)(1+a)-\delta-\varepsilon];\\
\tau_{\frac{n}{2}} &=& a(\gamma- \frac{n-1}{2});\\
\alpha_{\frac{n}{2}}\beta_{\frac{n}{2}} &=& \frac{n(n-1)}{2};\;\;\mathrm{and} \\
q_{\frac{n}{2}}&=& -\frac{n}{2}\left(2-n+\gamma\right)(1+a)-\delta-\varepsilon).
\end{eqnarray*}
and its associated eigenvalue $$E_{\frac{n}{2},e}=q_{\frac{n}{2}}+\Xi_{\frac{n}{2}}=n[(2-n+\gamma)(a+1)+\delta+\varepsilon]-q.$$
\end{cor}
\begin{proof}
By using the constraint equation~\eqref{cstr} and coefficient of $J_{+}$
\begin{eqnarray}
  \gamma+\delta+\varepsilon+\frac{3}{2}(2j-1) &=& \alpha+\beta+1+\frac{3}{2}(2j-1)\nonumber \\
   &=& \alpha+\beta+3j-\frac{1}{2}.\label{aggx}
\end{eqnarray}
so that
\begin{eqnarray}\label{qhd4}
 -H&=&\frac{1}{2}[ J_{+}J_{0}+J_{0}J_{+}]-\frac{(1+a)}{2}[J_{+}J_{-}+J_{-}J_{+}]+\frac{a}{2}[ J_{0}J_{-}+J_{-}J_{0}]\nonumber\\
 &&+[\gamma+\delta+\varepsilon
 +\frac{3}{2}(2j-1)]J_{+}+[(2j-1-\gamma)(1+a)-\delta-\varepsilon]J_{0}\nonumber\\
 &&+a[\gamma-\frac{(2j-1)}{2}]J_{-} +j[(2(1-j)+\gamma)(1+a)+\delta+\varepsilon]-q.
\end{eqnarray}
By setting $\alpha+\beta+3j-\frac{1}{2}=0$, the QES Heun operator in equation~\eqref{qhd4} reduces to an exactly solvable Heun Hamiltonian $H_e$ given by
\begin{eqnarray}\label{He}
- H_{e}&=&\frac{1}{2}[ J_{+}J_{0}+J_{0}J_{+}]-\frac{(1+a)}{2}[J_{+}J_{-}+J_{-}J_{+}]+\frac{a}{2}[ J_{0}J_{-}+J_{-}J_{0}]\nonumber\\
 &&+[(2j-1-\gamma)(1+a)-\delta-\varepsilon]J_{0}+a[\gamma-\frac{(2j-1)}{2}]J_{-}\nonumber\\
 &&+j[(2(1-j)+\gamma)(1+a)+\delta+\varepsilon]-q.
\end{eqnarray}
Now, setting $j=\frac{n}{2}$, \eqref{He} can be rewritten in the differential expression below as
 by expanding \eqref{He} in terms the product of generators $J_{+},J_{0},J_{-}$ in \eqref{Lgen} to get
\begin{multline}\label{qham1}
-H_{\frac{n}{2},e}=z(z-1)(z-a)D^{2}+[\rho_{\frac{n}{2}} z^{2} +\sigma_{\frac{n}{2}}  z+\tau_{\frac{n}{2}}]D+\alpha_{\frac{n}{2}}\beta_{\frac{n}{2}} z, \;\; (D=\vd/\vd z)
\end{multline}
with
\begin{eqnarray*}
 \rho_{\frac{n}{2}} &=& \frac{3(1-n)}{2};\\
 \sigma_{\frac{n}{2}} &=&( n-1-\gamma)(1+a)-\delta-\varepsilon;\\
\tau_{\frac{n}{2}} &=& a(\gamma- \frac{n-1}{2});\\
\alpha_{\frac{n}{2}}\beta_{\frac{n}{2}} &=& \frac{n(n-1)}{2};\;\;\mathrm{and} \\
q_{\frac{n}{2}}&=& -\frac{n}{2}\left[(n-\gamma)(1+a)-\delta-\varepsilon\right].
\end{eqnarray*}
The eigenvalue for the exactly solvable operator $-H_{\frac{n}{2},e}$ is $$E_{\frac{n}{2},e}=q_{\frac{n}{2}}+\Xi_{\frac{n}{2}}=n[(n-\gamma)(a+1)-\delta-\varepsilon]-q,\;\;n\in\mathbb{Z}.$$
\end{proof}

The group $G=SL(2,\mathbb{C})$ acts on the complex projective space $\mathbb{CP}^{1}=P(\mathbb{C}^{2})$ by linear fractional (M\"{o}bius) transformation
\begin{equation}\label{mob}
  \zeta\mapsto g\cdot \zeta = \frac{a\zeta +b}{c\zeta +d},\;\; g=\left(
                                                                   \begin{array}{cc}
                                                                     a & b \\
                                                                     c & d \\
                                                                   \end{array}
                                                                 \right),\;\; \det(g) =ad-bc=1.
\end{equation}
The finite-dimensional irreducible representation $\pi_{n,p}(g)$ of $SL(2,\mathbb{C})$ associated with above action in \eqref{mob} parameterized by pairs of non-negative integers $n,p$ may be described as follows. Let
$$\mathcal{P}^{n,p}:=\left\{P(z_1,z_2)| P(\alpha z_1,\alpha z_2)=\alpha^{n}\bar{\alpha}^{p}P(z_1,z_2),\;\;\alpha\in\mathbb{C}^{\times}\right\}$$
be the space of homogeneous polynomials of degrees $n$ and $p$ in the variables $z_1$ and $z_2$. The left regular representation of $SL(2,\mathbb{C})$ on $\mathcal{P}^{n,p}$ is given by
$$\pi_{n,p}(g)P(z_{1},z_{2})=P\left(g^{-1}\cdot \binom{z_{1}}{z_{2}}\right),\;\;\;\; g\in SL(2,\mathbb{C}).$$
(cf:~\cite{RW}, \S3-4:~54). The class of holomorphic irreducible representations of $SL(2,\mathbb{C})$ consists of representations of $\pi_{n}(g)$ ($\equiv \pi_{n,0}(g)$)
which acts on the space $\mathcal{P}^{n}$ of polynomials $\displaystyle P(z_{1},z_{2})=\sum_{k=0}^{n}c_{k}z_{1}^{k}z_{2}^{k}.$ For our purpose, we consider the following realization of $\mathcal{P}^{n,p}$ with the associated function $\widetilde{\psi}$ given by
$$P(z_1,z_2)=z_{1}^{n}\overline{z}_{1}^{p}\widetilde{\psi}(\frac{z_2}{z_1}).$$
Now put $z_1=1, z_2=\zeta.$ and define the function $\widetilde{\psi}$ as $\widetilde{\psi}(\zeta)=P(1,\zeta).$ Then, the left group action
\begin{equation}\label{mob2}
  g^{-1}\cdot \zeta =  g=\left(\begin{array}{cc} a & -c \\
                                                 -b& d \\
                                \end{array}   \right)\cdot \zeta=\left(\frac{a\zeta-c}{-b\zeta+d}\right),\;\; \zeta\in\mathbb{C}
\end{equation}
gives the one dimensional representation of $SL(2,\mathbb{C})$ on the space of polynomials $\mathcal{P}^{n}$ as
$$\pi_{n}(g)\widetilde{\psi}(\zeta)= (-b\zeta+d)^{n}\widetilde{\psi}\left(\frac{a\zeta-c}{-b\zeta+d}\right).$$
\begin{teo}[\cite{VVS}:~72] Let $\mathbf{B}$ be the subgroup of matrices of $G$ in the form $\left(\begin{array}{cc}
                                                                                                  a & b \\
                                                                                                          0 & a^{-1}
                                                                                                        \end{array}\right).$

$G$ acts on $\mathbb{CP}^1=\mathbb{C}\cup\{\infty\}$ by the M\"{o}bius transformation
$$m(g):~\zeta\mapsto \frac{a\zeta+b}{c\zeta+d}=m(g)(\zeta),\;\;g=\left(\begin{array}{cc}
                                                                                                        a & b \\
                                                                                                           c & d
                                \end{array}\right).$$
The group action $\zeta\mapsto g[\zeta]$ is given by
$$g[\zeta]=m((g^{-1})^{t})(\zeta)=\frac{d\zeta-c}{-b\zeta+a}.$$
For this action the stabilizer of $0$ is $\mathbf{B}$ so that $G/\mathbf{B}\cong \mathbb{CP}^{1}.$ We have a section defined on $\mathbb{C}$ given by $$s(z)=\left(\begin{array}{cc}
                                                                                                          1 & 1 \\
                                                                                                           -z & 1
                                \end{array}\right),\;\;s(z)[0]=z.$$
\end{teo}
\begin{proof}
Since $g\in \mathbf{B}\subset G=SL(2,\mathbb{C}), \det(g)=ad-bc=1$ and $$g^{-1}=\frac{1}{ad-bc}\left(\begin{array}{cc}
                                                                                                          d & -b \\
                                                                                                           -c & a
                                \end{array}\right)=\left(\begin{array}{cc}
                                                                                                          d & -b \\
                                                                                                           -c & a
                                \end{array}\right)$$
and
$$(g^{-1})^{\top}=\left(\begin{array}{cc}
                                                                                                          d & -c \\
                                                                                                           -b & a
                                \end{array}\right).$$
Thus by Mobi\"{u}s transformation $$g[\zeta]=m((g^{-1})^{\top})=\frac{d\zeta-c}{-b\zeta+a}.$$ Let $g$ be a $\mathbb{C}$-segment
\[s(z)=\left(\begin{array}{cc}
             1 & 1 \\
            -z & 1 \\
             \end{array}\right),\]
then
\[s(z)[\zeta]=\frac{\zeta+z}{-\zeta+1}=\frac{\zeta+z}{1-\zeta}\]
whence $s(z)[0]=z$ and $\vd s(z)[0]=\vd z.$ Therefore, the stabilizer of $0$ is $\mathbf{B}$ as required.
\end{proof}
\begin{rem} The measure $\vd x\vd y=\dfrac{i}{2}\vd z\vd \overline{z}$ is quasi-invariant  under the action of $G$ since $z'=\dfrac{az+b}{cz+d}$ gives $\vd z'=(cz+d)^{-2}\vd z$ and thus 
$$\vd x'\vd y'=|cz+d|^{-4}\vd x\vd y=\frac{i}{2}|cz+d|^{-4}\vd z\vd \overline{z}=:\vd g[z].$$
The quasi-invariant measure $\vd g[z]$ is required when handling differential expressions involving the Laplacian $\Delta:=\dfrac{\partial^{2}}{\partial x^{2}}+\dfrac{\partial^{2}}{\partial y^{2}}= 4\dfrac{\partial}{\partial z}\dfrac{\partial}{\partial \overline{z}}$ (see~\cite{WR}, \S11.3, p.223) but the Heun operator in this case does not involve $\Delta.$ However, considering the Borel set $\left(
                                      \begin{array}{cc}
                                        a & b \\
                                        0 & a^{-1} \\
                                      \end{array}
                                    \right)\in
\mathbf{B}\subset SL(2,\mathbb{C}), c=0,d=a^{-1}$ so that $\vd z'=(a^{-1})^{-2}\vd z=a^{2}\vd z,a=\varsigma_{n}^{k}=\exp\left(i\dfrac{2\pi k }{n}\right)\in \mathbb{C}^{\times}, k=0,\ldots,n-1, n\in\mathbb{Z}.$  The special case, suitable for defining the measure required for defining the Fourier transform will require that $n=2$ for which $a=\pm 1$ so that $\vd z'=\vd z$. It is also known (~\cite{MCM}, p.4) that some other subsets of $SL(2,\mathbb{C})$ up to conjugacy include $SL(2,\mathbb{C})$, compact subgroup $K=SU(2)$,unipotent matrices $N=\mathbb{C}$ and  diagonal matrices $A=\mathbb{C}^{\times}.$ It is also good to note here that $\mathbf{B}=AN.$ The choice of the measure on the Borel set $\mathbf{B}$ here is intuitive for the purpose of defining a suitable Fourier transform for obtaining a distributional solution of the Lie algebraic operator and its Green function.
\end{rem}

\section{Main Results}\label{GrHu}
In what follows, we present the distributional solution associated with Lie algebraic HDO, $H_{j}.$
\begin{teo} Consider the Lie algebraic HDO, $H_{j}\in U_{sl(2)}$. The a distributional solution $\Psi(z)$ associated to $H_{j}$ is given by
$$\Psi(z)=\sum_{k=0}^{\infty}c_{k}\delta^{(k)}(z).$$
provided that $c_{k}$ has either of these values
$$c_{k}=A(\varepsilon_{k,j}^{(1)}+\varepsilon_{k,j}^{(2)})^k+B(\varepsilon_{k,j}^{(1)}-\varepsilon_{k,j}^{(2)})^k$$
or
$$c_{k}=A[\eta_{k,j}^{(1)}+\eta_{k,j}^{(2)}]^{k}+B[\eta_{k,j}^{(1)}-\eta_{k,j}^{(2)}]^{k},$$
where,
\begin{eqnarray*}
 \varepsilon_{k,j}^{(1)}  &=& \frac{\rho_j(k+1)_{l_j+1}-\tau_j(k+1)_{l_j-1}}{2\alpha_j\beta_j(k)_{l_j}}, \\
 \varepsilon_{k,j}^{(2)} &=& \frac{\sqrt{\left(\rho_j(k+1)_{l_j+1}-\tau_j(k+1)_{l_j-1}\right)^2-4\alpha_j\beta_j(k)_{l_j}\big[(k+2)_{l_j+2}-a(k+2)_{l_j}\big]}}{2\alpha_j\beta_j(k)_{l_j}},\\
  \eta_{k,j}^{(1)} &=& -\frac{(k+1)_{l_j}\sigma_j}{E_j(k)_{l_j-1}}, \\
 \eta_{k,j}^{(2)}  &=& \frac{\sqrt{((k+1)_{l_j}\sigma_j)^2-4( E_j(k)_{l_j-1})([(1+a)(k+2)_{l_j+1}-a(k+2)_{l_j}])}}{E_j(k)_{l_j-1}}.
\end{eqnarray*}
Provided in each case, respectively,  $k\geq l_{j}$ and $k\geq l_{j-1}.$
\end{teo}
\begin{proof}
Let $\Psi(z)$ be the distributional solution of the Lie-algebraic Heun-type equation
 \begin{equation}\label{qhf1}
 (H_{j}-E_{j})\Psi(z)=0.
 \end{equation}
 Let $\Gamma:=\mathbb{CP}^1\setminus\{0,1,a,\infty\}$ be an open subset of the set of complex numbers, $\mathbb{C}$. Consider the Hilbert space $\mathfrak{H}=L^2(\Gamma,\vd\mu_{\omega}(z))$ and its Fourier transformed space $\widetilde{\mathfrak{H}}=L^2(\widetilde{\Gamma},\vd\mu_{\omega} (u)), \widetilde{\Gamma}\subset\widehat{\mathbb{C}}$. Let the character of $\mathbb{C}$ be given as
 $$\chi_{u}(z)=e^{iz\cdot u},\;\;z,u\in\mathbb{C}.$$
 The Fourier transform (isometry) $\mathscr{F}:~\mathfrak{H}\rightarrow\widetilde{\mathfrak{H}}$ is defined by
$$\langle f,\chi_{u}(z)\rangle_{\omega}=\mathscr{F}[f(u)]=\frac{1}{4\pi^{2}}\int_{\Gamma}f(z)\chi_{-u}(z)\,\omega(z)\vd z,\;\;\Gamma\subset \mathbb{C}$$
(Some references use $-iz\cdot u$ instead of $iz\cdot u,\;i=\sqrt{-1},z\in\mathbb{C}, \widehat{\mathbb{C}}$ is the dual of $\mathbb{C}$). The Fourier scaling factor here is $4\pi^2$ because $\mathbb{C}\cong \mathbb{R}^{2}$. Here, $\widehat{\mathbb{C}}$ is defined as
$$\widehat{\mathbb{C}}:=\{1,i,u, iu|u^2=0=i^2u^2, iu\cdot u=0=u\cdot iu\}.$$
In what follows in this section, a basic assumption that $\sigma_j,\tau_j, (a\neq 0,1)\in\mathbb{Z}$ is required. The weight function associated with $H_j$ is given by $$\omega(z)=z^{\rho_j-1}(z-1)^{\sigma_j-1}(z-a)^{\tau_j-1}=\sum_{m=0}^{\sigma_j-1}\sum_{n=0}^{\tau_j-1}h_{mn}^{j}z^{m+n+\rho_j-1}$$
where
$$h_{mn}^{j}=\binom{\sigma_j-1}{m}\binom{\tau_j-1}{n}(-1)^{\sigma_j+\tau_j+m+n}a^{\tau_j-n-1}.$$
The domain of $H_j$ denoted by $Dom(H_j)=\mathscr{C}_{c}^{\infty}(\Gamma)\subset \mathfrak{H}\subset \mathscr{D}'(\Gamma)$ where $\mathscr{D}'(\Gamma)$ is the space of distributions on $\Gamma$. This implies that $H_j$ admits distributional solutions, say $\displaystyle\Psi(z)=\sum_{k=0}^{\infty}c_{k}\delta^{(k)}(z), \;\; ^{(k)}=\frac{\vd^{k}}{\vd z^{k}}$. To obtain the values of $a_k$, one sets
$$\big\langle(H_j-E_j)\Psi(z),\chi_{u}(z)\big\rangle_\omega=0.$$
This yields
\begin{multline}\label{dsn}
  \sum_{k=0}^{\infty}\sum_{m=0}^{\sigma_j-1}\sum_{n=0}^{\tau_j-1} c_k h_{mn}^{j}\bigg[\bigg\{\left(-i\frac{\vd}{\vd u}\right)^{m+n+\rho_j+2}
  -(1+a)\left(-i\frac{\vd}{\vd u}\right)^{m+n+\rho_j+1}\\
  +a\left(-i\frac{\vd}{\vd u}\right)^{m+n+\rho_j}\bigg\}(-iu)^2
  +\bigg\{\rho_j\left(-i\frac{\vd}{\vd u}\right)^{m+n+\rho_j+1}+\sigma_j\left(-i\frac{\vd}{\vd u}\right)^{m+n+\rho_j}\\
  +\tau_j\bigg(-i\frac{\vd}{\vd u}\bigg)^{m+n+\rho_j-1}\bigg\}(-iu)
  +\alpha_j\beta_j\left(-i\frac{\vd}{\vd u}\right)^{m+n+\rho_j}-E_j\left(-i\frac{\vd}{\vd u}\right)^{m+n+\rho_j-1}\bigg](-iu)^{k}=0.
\end{multline}
Equation~\eqref{dsn} becomes
\begin{eqnarray}\label{dsn1-1}
 && \sum_{k=0}^{\infty}\sum_{m=0}^{\sigma_j-1}\sum_{n=0}^{\tau_j-1} c_k h_{mn}^{j}(-i)^{m+n+\rho_j+k}\bigg[\bigg\{(-i)^{4}\left(\frac{\vd}{\vd u}\right)^{m+n+\rho_j+2}u^{k+2}\nonumber\\
 && -(1+a)(-i)^3\left(\frac{\vd}{\vd u}\right)^{m+n+\rho_j+1}u^{k+2}
  +(-i)^{2} a\left(\frac{\vd}{\vd u}\right)^{m+n+\rho_j}u^{k+2}\bigg\}\nonumber\\
  &&+\bigg\{\rho_j(-i)^{2}\left(\frac{\vd}{\vd u}\right)^{m+n+\rho_j+1}u^{k+1}+\sigma_j(-i)\left(\frac{\vd}{\vd u}\right)^{m+n+\rho_j}u^{k+1}
  +\tau_j\bigg(\frac{\vd}{\vd u}\bigg)^{m+n+\rho_j-1}u^{k+1}\bigg\}\nonumber\\
  &&+\alpha_j\beta_j\left(\frac{\vd}{\vd u}\right)^{m+n+\rho_j}u^{k}-E_j(-i)^{-1}\left(\frac{\vd}{\vd u}\right)^{m+n+\rho_j-1}u^{k}\bigg]=0.
\end{eqnarray}
Equation~\eqref{dsn1-1} in terms of Pochhammer symbol is expressed as
\begin{multline}\label{dsn1}
  \sum_{k=0}^{\infty}\sum_{m=0}^{\sigma_j-1}\sum_{n=0}^{\tau_j-1} (-i)^{m+n+\rho_j+k}h_{mn}^{j}\\
  \times\bigg[\bigg\{(k+2)_{m+n+\rho_j+2}
  -i(1+a)(k+2)_{m+n+\rho_j+1}-a(k+2)_{m+n+\rho_j}\bigg\}c_{k-2} \\
  -\bigg\{\rho_j(k+1)_{m+n+\rho_j+1}-i(k+1)_{m+n+\rho_j}\sigma_j-\tau_j(k+1)_{m+n+\rho_j-1}\bigg\}c_{k-1}\\
  +[\alpha_j\beta_j(k)_{m+n+\rho_j}-i E_j(k)_{m+n+\rho_j-1}]c_k\bigg]=0.
\end{multline}
Here, $(k)_m=k(k-1)(k-2)\cdots(k-m+1).$ From equation~\eqref{dsn1}, one obtains 2 three-term recurrence relations by equating the real part and imaginary parts to zero respectively
\begin{multline}\label{dsn2}
 \bigg[\big[(k+2)_{m+n+\rho_j+2}-a(k+2)_{m+n+\rho_j}\big]c_{k-2} \\
  -\big[\rho_j(k+1)_{m+n+\rho_j+1}-\tau_j(k+1)_{m+n+\rho_j-1}\big]c_{k-1}\\
  +[\alpha_j\beta_j(k)_{m+n+\rho_j}]c_k\bigg]=0.
\end{multline}
and
\begin{equation}\label{dsn3}
-((1+a)(k+2)_{m+n+\rho_j+1}-a(k+2)_{m+n+\rho_j})c_{k-2}  +(k+1)_{m+n+\rho_j}\sigma_jc_{k-1} + E_j(k)_{m+n+\rho_j-1}c_k=0.
\end{equation}
 For convenience we set $l_j=m+n+\rho_j$ so that equation~\eqref{dsn2} and \eqref{dsn3} respectively become
 \begin{multline}\label{dsn3a}
 \bigg[\big[(k+2)_{l_j+2}-a(k+2)_{l_j}\big]c_{k-2} -\big[\rho_j(k+1)_{l_j+1}-\tau_j(k+1)_{l_j-1}\big]c_{k-1}  +[\alpha_j\beta_j(k)_{l_j}]c_k\bigg]=0
\end{multline}
and
\begin{equation}\label{dsn3b}
-[(1+a)(k+2)_{l_j+1}-a(k+2)_{l_j}]c_{k-2}  +(k+1)_{l_j}\sigma_jc_{k-1} + E_j(k)_{l_j-1}c_k=0.
\end{equation}
Solving the recurrence equation~\eqref{dsn3a} by setting  $c_{k}=s^{k}$, we have
$$\bigg[\big[(k+2)_{l_j+2}-a(k+2)_{l_j}\big]s^{k-2} -\big[\rho_j(k+1)_{l_j+1}-\tau_j(k+1)_{l_j-1}\big]s^{k-1}  +[\alpha_j\beta_j(k)_{l_j}]s^k\bigg]=0.$$
Dividing through by $s^{k-2}$ gives
$$\bigg[\big[(k+2)_{l_j+2}-a(k+2)_{l_j}\big] -\big[\rho_j(k+1)_{l_j+1}-\tau_j(k+1)_{l_j-1}\big]s  +[\alpha_j\beta_j(k)_{l_j}]s^2\bigg]=0.$$
Solving the resulting quadratic equation yields
\begin{multline*}
 s=\frac{\rho_j(k+1)_{l_j+1}-\tau_j(k+1)_{l_j-1}}{2\alpha_j\beta_j(k)_{l_j}}\pm \\
 \frac{\sqrt{\left(\rho_j(k+1)_{l_j+1}-\tau_j(k+1)_{l_j-1}\right)^2-4\alpha_j\beta_j(k)_{l_j}\big[(k+2)_{l_j+2}-a(k+2)_{l_j}\big]}}{2\alpha_j\beta_j(k)_{l_j}}
\end{multline*}
Setting $s=\varepsilon_{k,j}^{(1)}\pm \varepsilon_{k,j}^{(2)}$ so that
\begin{eqnarray*}
 \varepsilon_{k,j}^{(1)}  &=& \frac{\rho_j(k+1)_{l_j+1}-\tau_j(k+1)_{l_j-1}}{2\alpha_j\beta_j(k)_{l_j}} \\
 \varepsilon_{k,j}^{(2)} &=& \frac{\sqrt{\left(\rho_j(k+1)_{l_j+1}-\tau_j(k+1)_{l_j-1}\right)^2-4\alpha_j\beta_j(k)_{l_j}\big[(k+2)_{l_j+2}-a(k+2)_{l_j}\big]}}{2\alpha_j\beta_j(k)_{l_j}}
\end{eqnarray*}
we get the coefficient
$$c_{k}=A(\varepsilon_{k,j}^{(1)}+\varepsilon_{k,j}^{(2)})^{k}+B(\varepsilon_{k,j}^{(1)}-\varepsilon_{k,j}^{(2)})^{k}$$
where $A+B=1.$ We have non-trivial solution here provided $k\geq l_{j}$.

Alternatively, we consider the recurrence equation from the imaginary part given in equation~\eqref{dsn3b} and set $c_{k}=t^{k}$ so that
$$-[(1+a)(k+2)_{l_j+1}-a(k+2)_{l_j}]t^{k-2}  +(k+1)_{l_j}\sigma_jt^{k-1} + E_j(k)_{l_j-1}t^k=0.$$
Dividing through by $t^{k-2}$ we obtain
$$-[(1+a)(k+2)_{l_j+1}-a(k+2)_{l_j}]  +(k+1)_{l_j}\sigma_jt + E_j(k)_{l_j-1}t^2=0.$$
Solving the resulting quadratic equation we get
$$t=-\,\frac{(k+1)_{l_j}\sigma_j}{E_j(k)_{l_j-1}}\pm\frac{\sqrt{((k+1)_{l_j}\sigma_j)^2-4( E_j(k)_{l_j-1})([(1+a)(k+2)_{l_j+1}-a(k+2)_{l_j}])}}{E_j(k)_{l_j-1}}.$$
Let $t=\eta_{k,j}^{(1)}\pm\eta_{k,j}^{(2)}$ so that
\begin{eqnarray*}
  \eta_{k,j}^{(1)} &=& -\frac{(k+1)_{l_j}\sigma_j}{E_j(k)_{l_j-1}}, \\
 \eta_{k,j}^{(2)}  &=& \frac{\sqrt{((k+1)_{l_j}\sigma_j)^2-4( E_j(k)_{l_j-1})([(1+a)(k+2)_{l_j+1}-a(k+2)_{l_j}])}}{E_j(k)_{l_j-1}}.
\end{eqnarray*}
Thus, $$c_{k}=A[\eta_{k,j}^{(1)}+\eta_{k,j}^{(2)}]^{k}+B[\eta_{k,j}^{(1)}-\eta_{k,j}^{(2)}]^{k},$$
where $A+B=1.$ In this non-trivial solutions are obtained only when $k\geq l_{j}-1.$
The distributional solution $\Psi(z)$ is  therefore be written as
$$ \Psi(z)=\sum_{k=0}^{\infty}c_{k}\delta^{(k)}(z). $$
\end{proof}
 The next result leads to the computation of Green kernels $G\equiv G(z,w)$ and spectral shift function of the HDO $H_{\frac{n}{2},e}$. The technique adopted here is close in spirit with the one applied by~\cite{DGD}.
\begin{teo} Consider the exactly solvable Heun operator $H_{\frac{n}{2},e}$. By Fourier transform approach, the Green function associated with $H_{\frac{n}{2},e}$ is given by
$$  G(z,w)=\sum_{m=1}^{p} \frac{(i w)^{m-1}}{(m-1)!}\sum_{m=1}^{p}\sum_{k=0}^{\sigma_{\frac{n}{2}}-1}\sum_{l=0}^{\tau_{\frac{n}{2}}-1}h_{kl}a^{-l}(-1)^{m-1}
 \epsilon_{0} (m-1)!\delta(z).$$
with $\epsilon_{0}\in \mathbb{C}$ is a constant to be determined and $\displaystyle\phi(w)=\sum_{m=1}^{p}(iw)^{m-1}$. The spectral shift function in this case is
$$\xi^{\pm}(\lambda,H_{\frac{n}{2},e},H_{0})=G^{\pm}(w)\mathbb{H}(\lambda),w\in\Omega^{\pm}$$
where $\Omega^{\pm}\subset\mathbb{C}_{\pm}:=\{z\in\mathbb{C}|\pm\Im z>0\}$ and  $G^{\pm}(w)=G^{\pm}(w,w)$ given by
$$ G^{\pm}(w)= \sum_{m=1}^{p}\sum_{k=0}^{\sigma_{\frac{n}{2}}-1}\sum_{l=0}^{\tau_{\frac{n}{2}}-1}h_{kl}a^{-l}(-1)^{m-1}
  \epsilon_{0} \delta(w).$$
\end{teo}
\begin{proof}
Consider the non-homogeneous equation satisfied by Green function $G$
\begin{equation}\label{ghde}
  H_{\frac{n}{2},e}G=\delta(z-w)
\end{equation}
then
\begin{equation}\label{ghde2}
  \widehat{H_{\frac{n}{2},e}G}=\widehat{H_{\frac{n}{2},e}}\widehat{G}=\widehat{\delta(z-w)}.
\end{equation}
This yields
\begin{equation}\label{ghde3}
  \widehat{G}=\widehat{H_{\frac{n}{2},e}}^{-1}\widehat{\delta(z-w)}.
\end{equation}
Thus,
\begin{equation}\label{ghde4}
  G(z,w)=\mathscr{F}^{-1}\bigg[\widehat{H_{\frac{n}{2},e}}^{-1}\widehat{\delta(z-w)}\bigg]
\end{equation}
and trace value of its integral operator occurs when $z=w$ so that
\begin{equation}\label{ghde4a}
 G(w)=  G(w,w)=\mathscr{F}^{-1}\bigg[\widehat{H_{\frac{n}{2},e}}^{-1}\widehat{\delta(z-w)}\bigg]\bigg|_{z=w}
\end{equation}
The associated weight function for $H_{\frac{n}{2},e}$ is given by
\begin{eqnarray*}
  \omega(z) &=& z^{\rho_{\frac{n}{2}}-1}(z-1)^{\sigma_{\frac{n}{2}}-1}(z-a)^{\tau_{\frac{n}{2}}-1} \\
   &=& (-1)^{\sigma_{\frac{n}{2}}+\tau_{\frac{n}{2}}-2}a^{\tau_{\frac{n}{2}}-1}z^{\rho_{\frac{n}{2}}-1}\sum_{k=0}^{\sigma_{\frac{n}{2}}-1}\binom{\sigma_{\frac{n}{2}}-1}{k}z^k\sum_{l=0}^{\tau_{\frac{n}{2}}-1}\binom{\tau_{\frac{n}{2}}-1}{l}\left(\frac{z}{a}\right)^{l}  \\
   &=& (-1)^{\sigma_{\frac{n}{2}}+\tau_{\frac{n}{2}}-2}a^{\tau_{\frac{n}{2}}-1}z^{\rho_{\frac{n}{2}}-1}\sum_{k=0}^{\sigma_{\frac{n}{2}}-1}\sum_{l=0}^{\tau_{\frac{n}{2}}-1}\binom{\sigma_{\frac{n}{2}}-1}{k}\binom{\tau_{\frac{n}{2}}-1}{l}z^{k+l}a^{-l}   \\
   &=& (-1)^{\sigma_{\frac{n}{2}}+\tau_{\frac{n}{2}}-2}a^{\tau_{\frac{n}{2}}-1} \sum_{k=0}^{\sigma_{\frac{n}{2}}-1}\sum_{l=0}^{\tau_{\frac{n}{2}}-1}\binom{\sigma_{\frac{n}{2}}-1}{k}\binom{\tau_{\frac{n}{2}}-1}{l}z^{k+\rho_{\frac{n}{2}}+l-1}a^{-l}   \\
   &=& (-1)^{\sigma_{\frac{n}{2}}+\tau_{\frac{n}{2}}-2}a^{\tau_{\frac{n}{2}}-1} \sum_{k=0}^{\sigma_{\frac{n}{2}}-1}\sum_{l=0}^{\tau_{\frac{n}{2}}-1}h_{kl}z^{k+\rho_{\frac{n}{2}}+l-1}a^{-l},
\end{eqnarray*}
where $\displaystyle h_{kl}=\binom{\sigma_{\frac{n}{2}}-1}{k}\binom{\tau_{\frac{n}{2}}-1}{l}.$ In what follows, let $\chi_{-s},f:L^{2}(\Omega,\vd\mu_{w}(z))\rightarrow L^{2}(\Omega,\vd\mu_{w}(z)),$ $\Omega=\mathbb{CP}^{1}\setminus\{0,a,1,\infty\}$ then the Fourier transform
$$\widehat{f}:=\langle f,\chi_{s}\rangle_{\omega}=\int_{\Omega} f(z)\chi_{-s}(z)\vd\mu_{\omega}(z)$$
whence $\chi_{-s}(z)=e^{-is\cdot z}=\overline{\chi_{s}}(z)$ with $s=\sigma+i\tau, \tau=\Im s>0$ and Radon measure $\vd\mu_{\omega}(z)=\omega(z)\vd z.$ By using the inverse formula (cf:~\cite{VVS}, \S1.5, eq.(INV), p.8)
$$\reallywidecheck{f\,}=\mathscr{F}^{-1}\widehat{f}(z)=\int_{\Omega}\widehat{f}(s)\chi_{-s}(z)\vd\mu_{\omega}(z).$$

Computing the Fourier transforms of operators in \eqref{ghde4}, one obtains
\begin{eqnarray}
  \widehat{\delta(z-w)} &=& (-1)^{\sigma_{\frac{n}{2}}+\tau_{\frac{n}{2}}-2}a^{\tau_{\frac{n}{2}}-1} \sum_{k=0}^{\sigma_{\frac{n}{2}}-1}\sum_{l=0}^{\tau_{\frac{n}{2}}-1}h_{kl}a^{-l}\int_{\Omega}\delta(z-w)z^{k+\rho_{\frac{n}{2}}+l-1}\chi_{-s}(z)\vd z \nonumber \\
   &=& (-1)^{\sigma_{\frac{n}{2}}+\tau_{\frac{n}{2}}-2}a^{\tau_{\frac{n}{2}}-1} \sum_{k=0}^{\sigma_{\frac{n}{2}}-1}\sum_{l=0}^{\tau_{\frac{n}{2}}-1}h_{kl}a^{-l} w^{k+\rho_{\frac{n}{2}}+l-1}\chi_{-s}(w) \label{ghde5}
\end{eqnarray}
Next, it is obvious that
\begin{eqnarray*}
  \widehat{H_{\frac{n}{2},e}} &=& \mathscr{F}[H_{\frac{n}{2},e}] \\
   &=& \mathscr{F}\left[(z^3-(1+a)z^2+az)D^2+(\rho_{\frac{n}{2}}z^2+\sigma_{\frac{n}{2}}z+\tau_{\frac{n}{2}})D+\frac{n(n-1)}{2}z\right]
\end{eqnarray*}
can be simplified term-wisely as
\begin{eqnarray}
  \mathscr{F}[z^3D^2] &=& (-1)^{\sigma_{\frac{n}{2}}+\tau_{\frac{n}{2}}-2}a^{\tau_{\frac{n}{2}}-1}\sum_{k=0}^{\sigma_{\frac{n}{2}}-1}\sum_{l=0}^{\tau_{\frac{n}{2}}-1}h_{kl}a^{-l} \int_{\Omega} z^3D^2(z^{k+l+\rho_{\frac{n}{2}}-1}\chi_{-s}(z))\vd z\nonumber\\
   &=&(-1)^{\sigma_{\frac{n}{2}}+\tau_{\frac{n}{2}}-2}a^{\tau_{\frac{n}{2}}-1}\sum_{k=0}^{\sigma_{\frac{n}{2}}-1}\sum_{l=0}^{\tau_{\frac{n}{2}}-1}h_{kl}a^{-l}
   \int_{\Omega}[(k+l+\rho_{\frac{n}{2}}-1)(k+l+\rho_{\frac{n}{2}}-2)z^{k+l+\rho_{\frac{n}{2}}}\nonumber\\
   &&+2is(k+l+\rho_{\frac{n}{2}}-1)z^{k+l+\rho_{\frac{n}{2}}+1}-s^2z^{k+l+\rho_{\frac{n}{2}}+2}]\chi_{-s}(z)\vd z\nonumber\\
   &=&(-1)^{\sigma_{\frac{n}{2}}+\tau_{\frac{n}{2}}-2}a^{\tau_{\frac{n}{2}}-1}\sum_{k=0}^{\sigma_{\frac{n}{2}}-1}\sum_{l=0}^{\tau_{\frac{n}{2}}-1}h_{kl}a^{-l}
  [(k+l+\rho_{\frac{n}{2}}-1)(k+l+\rho_{\frac{n}{2}}-2)\left(\frac{\vd}{\vd s}\right)^{k+l+\rho_{\frac{n}{2}}}\nonumber\\
   &&+2is(k+l+\rho_{\frac{n}{2}}-1)\left(\frac{\vd}{\vd s}\right)^{k+l+\rho_{\frac{n}{2}}+1}-s^2\left(\frac{\vd}{\vd s}\right)^{k+l+\rho_{\frac{n}{2}}+2}];\label{ghde6}
\end{eqnarray}
\begin{eqnarray}
  -(1+a)\mathscr{F}[z^2D^2] &=& -(1+a)(-1)^{\sigma_{\frac{n}{2}}+\tau_{\frac{n}{2}}-2}a^{\tau_{\frac{n}{2}}-1}\sum_{k=0}^{\sigma_{\frac{n}{2}}-1}\sum_{l=0}^{\tau_{\frac{n}{2}}-1}h_{kl}a^{-l} \nonumber\\
  &&\times \int_{\Omega}z^{2}[(k+l+\rho_{\frac{n}{2}}-1)(k+l+\rho_{\frac{n}{2}}-2)z^{k+l+\rho_{\frac{n}{2}}-3}\nonumber\\
  &&+2is(k+l+\rho_{\frac{n}{2}}-1)z^{k+l+\rho_{\frac{n}{2}}-2}-s^2z^{k+l+\rho_{\frac{n}{2}}-1}]\chi_{-s}(z)\vd z\nonumber\\
   &=&  -(1+a)(-1)^{\sigma_{\frac{n}{2}}+\tau_{\frac{n}{2}}-2}a^{\tau_{\frac{n}{2}}-1}\nonumber\\
   &&\times \sum_{k=0}^{\sigma_{\frac{n}{2}}-1}\sum_{l=0}^{\tau_{\frac{n}{2}}-1}h_{kl}a^{-l} [(k+l+\rho_{\frac{n}{2}}-1)(k+l+\rho_{\frac{n}{2}}-2)\left(\frac{\vd}{\vd s}\right)^{k+l}\nonumber\\
&&+\rho_{\frac{n}{2}-1}+2is(k+l+\rho_{\frac{n}{2}}-1)\left(\frac{\vd}{\vd s}\right)^{k+l+\rho_{\frac{n}{2}}}-s^2\left(\frac{\vd}{\vd s}\right)^{k+l+\rho_{\frac{n}{2}}+1}];
   \label{ghde7}
\end{eqnarray}
\begin{eqnarray}
  a\mathscr{F}[zD^2] &=& a^{\tau_{\frac{n}{2}}}(-1)^{\sigma_{\frac{n}{2}}+\tau_{\frac{n}{2}}-2}\sum_{k=0}^{\sigma_{\frac{n}{2}}-1}\sum_{l=0}^{\tau_{\frac{n}{2}}-1}h_{kl}a^{-l}\int_{\Omega}z[(k+l+\rho_{\frac{n}{2}}-1)(k+l+\rho_{\frac{n}{2}}-2)z^{k+l+\rho_{\frac{n}{2}}-1}\nonumber\\
  &&+2is(k+l+\rho_{\frac{n}{2}}-1)z^{k+l+\rho_{\frac{n}{2}}}-s^2z^{k+l+\rho_{\frac{n}{2}}+1}]\chi_{-s}(z)\vd z;\nonumber\\
   &=& a^{\tau_{\frac{n}{2}}}(-1)^{\sigma_{\frac{n}{2}}+\tau_{\frac{n}{2}}-2}\sum_{k=0}^{\sigma_{\frac{n}{2}}-1}\sum_{l=0}^{\tau_{\frac{n}{2}}-1}h_{kl}a^{-l}[(k+l+\rho_{\frac{n}{2}}-1)(k+l+\rho_{\frac{n}{2}}-2)\left(\frac{\vd}{\vd s}\right)^{k+l+\rho_{\frac{n}{2}}}\nonumber\\
  &&+2is(k+l+\rho_{\frac{n}{2}}-1)\left(\frac{\vd}{\vd s}\right)^{k+l+\rho_{\frac{n}{2}}+1}-s^2\left(\frac{\vd}{\vd s}\right)^{k+l+\rho_{\frac{n}{2}}+2}];\label{ghde8}
\end{eqnarray}
\begin{eqnarray}
  \rho_{\frac{n}{2}}\mathscr{F}[z^{2}D] &=&  \rho_{\frac{n}{2}}a^{\tau_{\frac{n}{2}}-1}(-1)^{\sigma_{\frac{n}{2}}+\tau_{\frac{n}{2}}-2}\sum_{k=0}^{\sigma_{\frac{n}{2}}-1}\sum_{l=0}^{\tau_{\frac{n}{2}}-1}h_{kl}a^{-l} \int_{\Omega}z^{2}D(z^{k+l+\rho_{\frac{n}{2}}-1}\chi_{-s}(z))\vd z\nonumber\\
   &=& \rho_{\frac{n}{2}}a^{\tau_{\frac{n}{2}}-1}(-1)^{\sigma_{\frac{n}{2}}+\tau_{\frac{n}{2}}-2}\nonumber\\
   &&\times\sum_{k=0}^{\sigma_{\frac{n}{2}}-1}\sum_{l=0}^{\tau_{\frac{n}{2}}-1}h_{kl} a^{-l} \int_{\Omega}[(k+l+\rho_{\frac{n}{2}}-1)z^{k+l+\rho_{\frac{n}{2}}}+isz^{k+l+\rho_{\frac{n}{2}}+1}]\chi_{-s}(z)\vd z\nonumber\\
   &=& \rho_{\frac{n}{2}}a^{\tau_{\frac{n}{2}}-1}(-1)^{\sigma_{\frac{n}{2}}+\tau_{\frac{n}{2}}-2}\nonumber\\
   &&\times\sum_{k=0}^{\sigma_{\frac{n}{2}}-1}\sum_{l=0}^{\tau_{\frac{n}{2}}-1}h_{kl}a^{-l} [(k+l+\rho_{\frac{n}{2}}-1)\left(\frac{\vd}{\vd s}\right)^{k+l+\rho_{\frac{n}{2}}}+is\left(\frac{\vd}{\vd s}\right)^{k+l+\rho_{\frac{n}{2}}+1}];\nonumber\\
   \label{ghde9}
\end{eqnarray}
\begin{eqnarray}
  \sigma_{\frac{n}{2}}\mathscr{F}[zD] &=&  \sigma_{\frac{n}{2}}a^{\tau_{\frac{n}{2}}-1}(-1)^{\sigma_{\frac{n}{2}}+\tau_{\frac{n}{2}}-2}\sum_{k=0}^{\sigma_{\frac{n}{2}}-1}\sum_{l=0}^{\tau_{\frac{n}{2}}-1}h_{kl} a^{-l} \int_{\Omega}zD(z^{k+l+\rho_{\frac{n}{2}}-1}\chi_{-s}(z))\vd z\nonumber\\
   &=& \sigma_{\frac{n}{2}}a^{\tau_{\frac{n}{2}}-1}(-1)^{\sigma_{\frac{n}{2}}+\tau_{\frac{n}{2}}-2}\nonumber\\
   &&\times\sum_{k=0}^{\sigma_{\frac{n}{2}}-1}\sum_{l=0}^{\tau_{\frac{n}{2}}-1}h_{kl} a^{-l} \int_{\Omega}[(k+l+\rho_{\frac{n}{2}}-1)z^{k+l+\rho_{\frac{n}{2}}-1}+isz^{k+l+\rho_{\frac{n}{2}}}]\chi_{-s}(z)\vd z\nonumber\\
   &=& \sigma_{\frac{n}{2}}a^{\tau_{\frac{n}{2}}-1}(-1)^{\sigma_{\frac{n}{2}}+\tau_{\frac{n}{2}}-2}\nonumber\\
   &&\times\sum_{k=0}^{\sigma_{\frac{n}{2}}-1}\sum_{l=0}^{\tau_{\frac{n}{2}}-1}h_{kl} a^{-l} [(k+l+\rho_{\frac{n}{2}}-1)\left(\frac{\vd}{\vd s}\right)^{k+l+\rho_{\frac{n}{2}}-1}+is\left(\frac{\vd}{\vd s}\right)^{k+l+\rho_{\frac{n}{2}}}];\nonumber\\ \label{ghde10}
\end{eqnarray}
\begin{eqnarray}
  \tau_{\frac{n}{2}}\mathscr{F}[zD] &=&  \tau_{\frac{n}{2}}a^{\tau_{\frac{n}{2}}-1}(-1)^{\sigma_{\frac{n}{2}}+\tau_{\frac{n}{2}}-2}\sum_{k=0}^{\sigma_{\frac{n}{2}}-1}\sum_{l=0}^{\tau_{\frac{n}{2}}-1}h_{kl} a^{-l} \int_{\Omega}D(z^{k+l+\rho_{\frac{n}{2}}-1}\chi_{-s}(z))\vd z\nonumber\\
   &=& \tau_{\frac{n}{2}}a^{\tau_{\frac{n}{2}}-1}(-1)^{\sigma_{\frac{n}{2}}+\tau_{\frac{n}{2}}-2}\nonumber\\
   &&\times\sum_{k=0}^{\sigma_{\frac{n}{2}}-1}\sum_{l=0}^{\tau_{\frac{n}{2}}-1}h_{kl}a^{-l} \int_{\Omega}[(k+l+\rho_{\frac{n}{2}}-1)z^{k+l+\rho_{\frac{n}{2}}-2}+isz^{k+l+\rho_{\frac{n}{2}}-1}]\chi_{-s}(z)\vd z\nonumber\\
   &=& \tau_{\frac{n}{2}}a^{\tau_{\frac{n}{2}}-1}(-1)^{\sigma_{\frac{n}{2}}+\tau_{\frac{n}{2}}-2}\nonumber\\
   &&\times\sum_{k=0}^{\sigma_{\frac{n}{2}}-1}\sum_{l=0}^{\tau_{\frac{n}{2}}-1}h_{kl} a^{-l} [(k+l+\rho_{\frac{n}{2}}-1)\left(\frac{\vd}{\vd s}\right)^{k+l+\rho_{\frac{n}{2}}-2}+is\left(\frac{\vd}{\vd s}\right)^{k+l+\rho_{\frac{n}{2}}-1}];\nonumber\\\label{ghde11}
  \end{eqnarray}
  \begin{eqnarray}
  \frac{n(n-1)}{2}\mathscr{F}[z] &=& \frac{n(n-1)}{2} a^{\tau_{\frac{n}{2}}-1}(-1)^{\sigma_{\frac{n}{2}}+\tau_{\frac{n}{2}}-2}\sum_{k=0}^{\sigma_{\frac{n}{2}}-1}\sum_{l=0}^{\tau_{\frac{n}{2}}-1}h_{kl}a^{-l} \int_{\Omega}z^{k+l+\rho_{\frac{n}{2}}-1}\chi_{-s}(z)\vd z\nonumber\\
   &=& \frac{n(n-1)}{2}a^{\tau_{\frac{n}{2}}-1}(-1)^{\sigma_{\frac{n}{2}}+\tau_{\frac{n}{2}}-2}
   \sum_{k=0}^{\sigma_{\frac{n}{2}}-1}\sum_{l=0}^{\tau_{\frac{n}{2}}-1}h_{kl} a^{-l}\left(\frac{\vd}{\vd s}\right)^{k+l+\rho_{\frac{n}{2}}-1}.\label{ghde12}
\end{eqnarray}
Adding up equations~\eqref{ghde6}-\eqref{ghde12} yields
\begin{eqnarray}
  \widehat{H_{\frac{n}{2},e}} &=& a^{\tau_{\frac{n}{2}}-1}(-1)^{\sigma_{\frac{n}{2}}+\tau_{\frac{n}{2}}-2}
   \sum_{k=0}^{\sigma_{\frac{n}{2}}-1}\sum_{l=0}^{\tau_{\frac{n}{2}}-1}h_{kl}a^{-l}\nonumber\\ &&\cdot\bigg\{[(k+l+\rho_{\frac{n}{2}}-1)[(1+a)(k+l+\rho_{\frac{n}{2}}-2-2is)+\rho_{\frac{n}{2}}]+is\sigma_{\frac{n}{2}}]\left(\frac{\vd}{\vd s}\right)^{k+l+\rho_{\frac{n}{2}}} \nonumber\\
  &&+[2is(1+a)(k+l+\rho_{\frac{n}{2}}-1)+\rho_{\frac{n}{2}}is-s^2]\left(\frac{\vd}{\vd s}\right)^{k+l+\rho_{\frac{n}{2}}+1}-2s^{2}\left(\frac{\vd}{\vd s}\right)^{k+l+\rho_{\frac{n}{2}}+2}\nonumber\\
  &&+[(k+l+\rho_{\frac{n}{2}}-1)(\sigma_{\frac{n}{2}}-(1+a)(k+l+\rho_{\frac{n}{2}}-2))+is\tau_{\frac{n}{2}}+\frac{n(n-1)}{2}]\left(\frac{\vd}{\vd s}\right)^{k+l+\rho_{\frac{n}{2}}-1}\nonumber\\
  && +\tau_{\frac{n}{2}}(k+l+\rho_{\frac{n}{2}}-1)\left(\frac{\vd}{\vd s}\right)^{k+l+\rho_{\frac{n}{2}}-2}\bigg\}.\label{ghde13}
\end{eqnarray}
Let $m_{kl}=k+l+\rho_{\frac{n}{2}}$ and $c_{n}=a^{\tau_{\frac{n}{2}}-1}(-1)^{\sigma_{\frac{n}{2}}+\tau_{\frac{n}{2}}-2}$ then the expression in \eqref{ghde13} becomes
\begin{eqnarray*}
  \widehat{H_{\frac{n}{2},e}} &=& c_{n}
   \sum_{k=0}^{\sigma_{\frac{n}{2}}-1}\sum_{l=0}^{\tau_{\frac{n}{2}}-1}h_{kl}a^{-l}\left\{[(m_{kl}-1)[(1+a)(m_{kl}-2-2is)+\rho_{\frac{n}{2}}]+is\sigma_{\frac{n}{2}}]\left(\frac{\vd}{\vd s}\right)^{m_{kl}} \right.\nonumber\\
  &&+[2is(1+a)(m_{kl}-1)+\rho_{\frac{n}{2}}is-s^2]\left(\frac{\vd}{\vd s}\right)^{m_{kl}+1}-2s^{2}\left(\frac{\vd}{\vd s}\right)^{m_{kl}+2}\nonumber\\
  &&+[(m_{kl}-1)(\sigma_{\frac{n}{2}}-(1+a)(m_{kl}-2))+is\tau_{\frac{n}{2}}+\frac{n(n-1)}{2}]\left(\frac{\vd}{\vd s}\right)^{m_{kl}-1}\nonumber\\
  &&\left. +\tau_{\frac{n}{2}}(m_{kl}-1)\left(\frac{\vd}{\vd s}\right)^{m_{kl}-2}\right\}
\end{eqnarray*}
which further becomes
\begin{eqnarray*}
 \widehat{H_{\frac{n}{2},e}} &=& c_{n}
   \sum_{m_{kl}=0}^{\sigma_{\frac{n}{2}}-\rho_{\frac{n}{2}}-l-1}\;\sum_{l=0}^{\tau_{\frac{n}{2}}-1}h_{kl}a^{-l}\left\{[(m_{kl}-1)[(1+a)(m_{kl}-2-2is)+\rho_{\frac{n}{2}}]+is\sigma_{\frac{n}{2}}]\left(\frac{\vd}{\vd s}\right)^{m_{kl}} \right.\\
  &&+[2is(1+a)(m_{kl}-1)+\rho_{\frac{n}{2}}is-s^2]\left(\frac{\vd}{\vd s}\right)^{m_{kl}+1}-2s^{2}\left(\frac{\vd}{\vd s}\right)^{m_{kl}+2}\\
  &&+[(m_{kl}-1)(\sigma_{\frac{n}{2}}-(1+a)(m_{kl}-2))+is\tau_{\frac{n}{2}}+\frac{n(n-1)}{2}]\left(\frac{\vd}{\vd s}\right)^{m_{kl}-1}\\
  &&\left. +\tau_{\frac{n}{2}}(m_{kl}-1)\left(\frac{\vd}{\vd s}\right)^{m_{kl}-2}\right\}.
\end{eqnarray*}
By index shifting
\begin{eqnarray}\label{ghde14}
\widehat{H_{\frac{n}{2},e}}&=& c_{n}
   \sum_{m_{kl}=2}^{\sigma_{\frac{n}{2}}-\rho_{\frac{n}{2}}-l+1}\;\sum_{l=0}^{\tau_{\frac{n}{2}}-1}h_{kl}a^{-l}\left\{\left[(m_{kl}-1)[(1+a)(m_{kl}-2-2is)+\rho_{\frac{n}{2}}]+is\sigma_{\frac{n}{2}}\right.\right.\nonumber\\
   &&+\tau_{\frac{n}{2}}(m_{kl}+1)]\left(\frac{\vd}{\vd s}\right)^{m_{kl}}
  +[\left(2(1+a)(m_{kl}-1)+(\rho_{\frac{n}{2}}+\tau_{\frac{n}{2}})\right)is-s^2\nonumber\\
 && \left.+(m_{kl}+1)(\sigma_{\frac{n}{2}}-(1+a)m_{kl})+\frac{n(n-1)}{2}]\left(\frac{\vd}{\vd s}\right)^{m_{kl}+1}-2s^{2}\left(\frac{\vd}{\vd s}\right)^{m_{kl}+2}\right\}\nonumber\\
\end{eqnarray}
and by \eqref{ghde5}
\begin{equation}\label{ghde15}
  \widehat{\delta(z-w)}=c_{n}\sum_{m_{kl}=1}^{\sigma_{\frac{n}{2}}-\rho_{\frac{n}{2}}-l}\;\sum_{l=0}^{\tau_{\frac{n}{2}}-1}h_{kl}a^{-l}w^{m_{kl}-1}\chi_{-s}(w).
\end{equation}
Let
\begin{eqnarray}\label{ghde16}
  \Delta(m_{kl},l,\frac{\vd}{\vd s})&\equiv&\sum_{m_{kl}=2}^{\sigma_{\frac{n}{2}}-\rho_{\frac{n}{2}}-l+1}\;\sum_{l=0}^{\tau_{\frac{n}{2}}-1}h_{kl}a^{-l}\bigg\{[(m_{kl}-1)[(1+a)(m_{kl}-2-2is)+\rho_{\frac{n}{2}}]+is\sigma_{\frac{n}{2}}\nonumber\\
  && +\tau_{\frac{n}{2}}(m_{kl}+1)]
  +[\left(2(1+a)(m_{kl}-1)+(\rho_{\frac{n}{2}}+\tau_{\frac{n}{2}})\right)is-s^2\nonumber\\
  &&+(m_{kl}+1)(\sigma_{\frac{n}{2}}-(1+a)m_{kl})+\frac{n(n-1)}{2}]\frac{\vd}{\vd s}-2s^{2}\left(\frac{\vd}{\vd s}\right)^{2}\bigg\}.
\end{eqnarray}
Then
$$\widehat{H_{\frac{n}{2},e}}=c_{n}\Delta(m_{kl},l,\frac{\vd}{\vd s})\left(\frac{\vd}{\vd s}\right)^{m_{kl}}.$$
Now, the  Fourier of Green function is given by
\begin{eqnarray}\label{gde17}
 \widehat{ G}&=&\Delta(m_{kl},l,\frac{\vd}{\vd s})^{-1}\left(\frac{\vd}{\vd s}\right)^{-m_{kl}}\sum_{m_{kl}=1}^{\sigma_{\frac{n}{2}}-\rho_{\frac{n}{2}}-l}\;\sum_{l=0}^{\tau_{\frac{n}{2}}-1}h_{kl}a^{-l}w^{m_{kl}-1}\chi_{-s}(w)\nonumber\\
  &=&\Delta(m_{kl},l,\frac{\vd}{\vd s})^{-1}\sum_{m_{kl}=1}^{\sigma_{\frac{n}{2}}-\rho_{\frac{n}{2}}-l}\;\sum_{l=0}^{\tau_{\frac{n}{2}}-1}h_{kl}a^{-l}w^{m_{kl}-1}\underbrace{\int\cdots\int}_{m_{kl}}\chi_{-s}(w)\underbrace{\vd s\cdots\vd s}_{m_{kl}}\nonumber\\
  &=&\Delta(m_{kl},l,\frac{\vd}{\vd s})^{-1}\sum_{m_{kl}=1}^{\sigma_{\frac{n}{2}}-\rho_{\frac{n}{2}}-l}\;\sum_{l=0}^{\tau_{\frac{n}{2}}-1}h_{kl}a^{-l}\underbrace{\int\cdots\int}_{m_{kl}-1} \delta(w)\underbrace{\vd s\cdots\vd s}_{m_{kl}-1} \nonumber\\
  &=&\Delta(m_{kl},l,\frac{\vd}{\vd s})^{-1}\sum_{m_{kl}=1}^{\sigma_{\frac{n}{2}}-\rho_{\frac{n}{2}}-l}\;\sum_{l=0}^{\tau_{\frac{n}{2}}-1}h_{kl}a^{-l}w^{m_{kl}-1}\frac{s^{m_{kl}-1}}{(m_{kl}-1)!}.\nonumber\\
\end{eqnarray}
Let the operator $\Delta(m_{kl},l,\frac{\vd}{\vd s})$ be written in a factorable form
$$\Delta(m_{kl},l,\frac{\vd}{\vd s})=\sum \epsilon_{0}+\epsilon_{1}\frac{\vd}{\vd s}+\epsilon_{2}\left(\frac{\vd}{\vd s}\right)^{2}\equiv\sum \left(\eta_{+}+\frac{\vd}{\vd s}\right)\left(\eta_{-}+\frac{\vd}{\vd s}\right)$$
where $\displaystyle\sum\equiv\sum_{m_{kl}=1}^{\sigma_{\frac{n}{2}}-\rho_{\frac{n}{2}}-l}\;\sum_{l=0}^{\tau_{\frac{n}{2}}-1}h_{kl}a^{-l}$ and
\begin{eqnarray*}
  \epsilon_{0} &=&  (m_{kl}-1)[(1+a)(m_{kl}-2-2is)+\rho_{\frac{n}{2}}]+is\sigma_{\frac{n}{2}}+\tau_{\frac{n}{2}}(m_{kl}+1)\\
  \epsilon_{1} &=&  \left(2(1+a)(m_{kl}-1)+(\rho_{\frac{n}{2}}+\tau_{\frac{n}{2}})\right)is-s^2+(m_{kl}+1)(\sigma_{\frac{n}{2}}-(1+a)m_{kl})+\frac{n}{2}(n-1)\\
  \epsilon_{2} &=& -2s^{2}
\end{eqnarray*}
with
\begin{equation*}
  \eta_{\pm}=\frac{-\epsilon_{1}\pm\sqrt{\epsilon_{1}^{2}-4\epsilon_{0}\epsilon_{2}}}{2\epsilon_{0}}.
\end{equation*}
Let
$$\phi(s,w)=\sum_{m_{kl}=1}^{\sigma_{\frac{n}{2}}-\rho_{\frac{n}{2}}-l}\;\sum_{l=0}^{\tau_{\frac{n}{2}}-1}h_{kl}a^{-l}w^{m_{kl}-1}\frac{s^{m_{kl}-1}}{(m_{kl}-1)!}.$$
By (\cite{AF}, Chapter 16, short methods (c), p.99), given function of operators
$$F(\frac{\vd}{\vd s})=\prod_{p=0}^{n}(\eta_{p}+\frac{\vd}{\vd s})\equiv\sum_{k=0}^{n}a_{k}\left(\frac{\vd}{\vd s}\right)^{k}$$ with constant coefficients $a_k$ then
\begin{equation}\label{gde16b}
  \frac{1}{F(\frac{\vd}{\vd s})}s^{m}=\left(\sum_{k=0}^{m}a_{k}\left(\frac{\vd}{\vd s}\right)^{k}\right)  s^{m}.
\end{equation}
Then by equations~\eqref{gde16b} and \eqref{gde17} one gets
\begin{equation*}
  \Delta(m_{kl},l,\frac{\vd}{\vd s})^{-1}\phi(w,s)=\sum_{m_{kl}=1}^{\sigma_{\frac{n}{2}}-\rho_{\frac{n}{2}}-l}\;\sum_{l=0}^{\tau_{\frac{n}{2}}-1}h_{kl}a^{-l}\frac{w^{m_{kl}-1}}{(m_{kl}-1)!}\left(\sum \epsilon_{0}+\epsilon_{1}\left(\frac{\vd}{\vd s}\right)+\epsilon_{2}\left(\frac{\vd}{\vd s}\right)^{2}\right)s^{m_{kl}-1}.
\end{equation*}
Thus, $\widehat{G}$ associated with the Heun differential operator is
$$\widehat{G}=\sum\frac{w^{m_{kl}-1}}{(m_{kl}-1)!}\left(\sum \epsilon_{0}s^{m_{kl}-1}+\epsilon_{1}(m_{kl}-1)s^{m_{kl}-2}+\epsilon_{2}(m_{kl}-1)(m_{kl}-2)s^{m_{kl}-3}\right).$$
$s=\sigma+i\tau$ and $-H_0=D^2$. It can be recalled that given a polynomial $P:\mathbb{C}\rightarrow\mathbb{C}, P(\frac{\vd}{\vd z})\delta(z)=\mathscr{F}^{-1}[P(-is)].$ Thus, one obtains the Green function $G(z,w)=\reallywidecheck{\widehat{G}\;}$ by
\begin{multline}\label{ddff}
  G(z,w)
  =\sum i^{m_{kl}-1}\frac{w^{m_{kl}-1}}{(m_{kl}-1)!}\\
  \times\left(\sum\reallywidecheck{\epsilon_{0}(-is)^{m_{kl}-1}}-i(m_{kl}-1)\epsilon_{1}\reallywidecheck{(-is)^{m_{kl}-2}}-(m_{kl}-1)(m_{kl}-2)\epsilon_{2}\reallywidecheck{(-is)^{m_{kl}-3}}\right).
\end{multline}
For convenience, let $m=m_{kl}, c_{m}=\frac{i^{m_{kl}}}{m_{kl}!}, d_{m}=-im_{kl}, e_{m}=-m_{kl}(m_{kl}-1)$,
$\displaystyle p=\sigma_{\frac{n}{2}}-\rho_{\frac{n}{2}}-l, \sum\equiv\sum_{m=1}^{p}.$
By using the property of Dirac delta function derivatives (See~\cite{KRP}, \S 9.11 (4):~243)
 \begin{equation*}
   w^{n}\delta^{(m)}(w)=\left\{\begin{array}{cc}
                                 0 & m<n \\
                                 {} & {}\\
                  \displaystyle    (-1)^{n}\frac{m!}{(m-n)!}\delta^{(m-n)}(w)& m\geq n.
                               \end{array}
   \right.
 \end{equation*}
one evaluates the inverse Fourier transforms
\begin{eqnarray}
  \reallywidecheck{(-is)^{m-1}} &=& \int_{\Omega}(-is)^{m-1}e^{s\cdot z}\sum_{k=0}^{\sigma_{\frac{n}{2}}-1}\sum_{l=0}^{\tau_{\frac{n}{2}}-1}h_{kl}a^{-l}z^{m-1}\vd s \nonumber\\
   &=& \sum_{k=0}^{\sigma_{\frac{n}{2}}-1}\sum_{l=0}^{\tau_{\frac{n}{2}}-1}h_{kl}a^{-l}z^{m-1}\int_{\Omega}(-is)^{m-1}e^{s\cdot z}\vd s \nonumber\\
  &=& \sum_{k=0}^{\sigma_{\frac{n}{2}}-1}\sum_{l=0}^{\tau_{\frac{n}{2}}-1}h_{kl}a^{-l}z^{m-1}\delta^{(m-1)}(z)\nonumber\\
  &=& \sum_{k=0}^{\sigma_{\frac{n}{2}}-1}\sum_{l=0}^{\tau_{\frac{n}{2}}-1}h_{kl}a^{-l}(-1)^{m-1}(m-1)!\delta(z) \label{dff9a}\\
  \reallywidecheck{(-is)^{m-2}} &=& \sum_{k=0}^{\sigma_{\frac{n}{2}}-1}\sum_{l=0}^{\tau_{\frac{n}{2}}-1}h_{kl}a^{-l}z^{m-1}\delta^{(m-2)}(z)=0 \label{dff9b} \\
  \reallywidecheck{(-is)^{m-3}} &=& \sum_{k=0}^{\sigma_{\frac{n}{2}}-1}\sum_{l=0}^{\tau_{\frac{n}{2}}-1}h_{kl}a^{-l}z^{m-1}\delta^{(m-3)}(z)=0.\nonumber\\
      \label{dff9c}
\end{eqnarray}
Re-substituting equations~\eqref{dff9a}-\eqref{dff9c} into equation~\eqref{ddff}, one gets
 \begin{equation}\label{ddff1}
  G(z,w)=\sum_{m=1}^{p} \frac{(i w)^{m-1}}{(m-1)!}\sum_{m=1}^{p}\sum_{k=0}^{\sigma_{\frac{n}{2}}-1}\sum_{l=0}^{\tau_{\frac{n}{2}}-1}h_{kl}a^{-l}(-1)^{m-1}
 \epsilon_{0} (m-1)!\delta(z).
\end{equation}
In terms of the eigenvalue $E_{\frac{n}{2},e},$ the Green function $G(E_{\frac{n}{2},e},z,w)$ of $H_{\frac{n}{2},e}$ with $Dom(H_{\frac{n}{2},e})\in\mathscr{C}_{c}^{\infty}(\Omega)$ is defined by the integral $$(H_{\frac{n}{2},e}-E_{\frac{n}{2},e})^{-1}f(z)=\int_{\Omega}G(E_{\frac{n}{2},e},z,w)f(w)\vd w, \;\;\;\Im E_{\frac{n}{2},e}\neq 0.$$
In what follows, it is of particular interest, to obtain the Green function at the point $z=w\in\Omega^{\pm}$ (the upper and lower complex half-plane excluding the singular points $0,1,a,\infty$). By setting $z=w$, one obtains
\begin{eqnarray*}
  G(E_{\frac{n}{2},e}, w)=G(E_{\frac{n}{2},e}, w,w)   &=&\left\{\begin{array}{cc}
                      G^{+}(E_{\frac{n}{2},e}, w,w)&, w\in\Omega^{+} \\
                       &\\
                      G^{-}(E_{\frac{n}{2},e}, w,w)&, w\in \Omega^{-}
                     \end{array}
  \right.
\end{eqnarray*}
where, $G^{\pm}(E_{\frac{n}{2},e}, w,w)=G^{\pm}(E_{\frac{n}{2},e},w)$ can now be written in the form
\begin{equation}\label{grsum}
  G^{\pm}(E_{\frac{n}{2},e},w)= \sum_{m=1}^{p} \frac{(iw)^{m-1}}{(m-1)!}\sum_{m=1}^{p}\sum_{k=0}^{\sigma_{\frac{n}{2}}-1}\sum_{l=0}^{\tau_{\frac{n}{2}}-1}h_{kl}a^{-l}\frac{(-1)^{m-1}}{E_{\frac{n}{2},e}}
  \epsilon_{0} (m-1)!\delta(\pm w).
\end{equation}
It is obvious that $\displaystyle\sum_{m=1}^{p} \frac{(iw)^{m-1}}{(m-1)!}$  is a $(p-1)$th approximation of $e^{iw}$. By setting $\displaystyle\phi(w)=\sum_{m=1}^{p}(iw)^{m-1}$ and following ~(\cite{JAB}, Definition~3, Eq.(9), p.3), then equation~\eqref{grsum}  may be written as
\begin{equation}\label{grsum1}
 G^{\pm}(E_{\frac{n}{2},e},w)= \sum_{m=1}^{p}\sum_{k=0}^{\sigma_{\frac{n}{2}}-1}\sum_{l=0}^{\tau_{\frac{n}{2}}-1}h_{kl}a^{-l}\frac{(-1)^{m-1}}{E_{\frac{n}{2},e}}
   \epsilon_{0} \phi(w)\delta(\pm w).
\end{equation}
It should be remarked here that $\delta(w)$ is an even function for $\ell\in\mathbb{Z}_{0}^{+}$, thus
$$\delta^{(m)}(\pm w)=\left\{\begin{array}{cc}
                           +\delta^{(m)}(w),   & m=2\ell, \\
                              &  \\
                             - \delta^{(m)}(w), & m=2\ell+1 ,
                           \end{array}\right.$$
and
 \begin{equation}\label{grsu}
   \phi(w)\delta^{(m)}(w)=(-1)^{m}\sum_{r=0}^{m}\binom{m}{r}(-1)^{r}\phi^{(m-r)}(0)\delta^{(r)}(w),\;\;\binom{m}{r}=\frac{m!}{r!(m-r)!}.
 \end{equation}
 From these data, the expression ~\eqref{grsum1} for Green function becomes
\begin{eqnarray}\label{grsum2}
  G^{\pm}(E_{\frac{n}{2},e},w)&=& \sum_{m=1}^{p}\sum_{k=0}^{\sigma_{\frac{n}{2}}-1}\sum_{l=0}^{\tau_{\frac{n}{2}}-1}h_{kl}a^{-l}\frac{(-1)^{m-1}}{E_{\frac{n}{2},e}}\epsilon_{0} \phi(w)\delta(w)\nonumber\\
&=&\sum_{m=1}^{p}\sum_{k=0}^{\sigma_{\frac{n}{2}}-1}\sum_{l=0}^{\tau_{\frac{n}{2}}-1}h_{kl}a^{-l}\frac{(-1)^{m-1}}{E_{\frac{n}{2},e}}
  \epsilon_{0} \phi(0)\delta(w),
\end{eqnarray}
where $\phi(0)=1$ so that  equation~\eqref{grsum2} becomes
\begin{equation}\label{grsum4}
  G^{\pm}(E_{\frac{n}{2},e},w)= \sum_{m=1}^{p}\sum_{k=0}^{\sigma_{\frac{n}{2}}-1}\sum_{l=0}^{\tau_{\frac{n}{2}}-1}h_{kl}a^{-l}\frac{(-1)^{m-1}}{E_{\frac{n}{2},e}}
  \epsilon_{0} \delta(w).
\end{equation}
It should be remarked here that $$G^{\pm}(\cdot,w)\in\mathscr{C}_{c}^{\infty}(\Omega)\subset L^{2}(\Omega^{+}, \vd_{\omega}\mu(w))\oplus L^{2}(\Omega^{-}, \vd_{\omega}\mu(w)).$$ By setting the eigenvalue $E_{\frac{n}{2}}=\lambda+i\epsilon$ were $\lambda=\Re E_{\frac{n}{2}}$ and $\epsilon=\Im E_{\frac{n}{2}}$ one defines the SSF using the Privalov's representation (cf:~\cite{GES}, \S (4.8):~760) as
\begin{equation}\label{ghde17}
  \xi^{\pm}(\lambda,H_{\frac{n}{2},e},H_{0})=\frac{1}{\pi}\lim_{\epsilon\downarrow 0}\textrm{Arg}\; G^{\pm}(\lambda+i\epsilon,w).
\end{equation}
Further simplification leads to
\begin{eqnarray}
   \xi^{\pm}(\lambda,H_{\frac{n}{2},e},H_{0}) &=& \frac{1}{\pi}G^{\pm}(w,w)\lim_{\epsilon\downarrow 0}\textrm{Arg}\frac{1}{\lambda+i\epsilon} \nonumber\\
   &=& \frac{1}{\pi}G^{\pm}(w,w)\lim_{\epsilon\downarrow 0}\textrm{Arg}\frac{\lambda-i\epsilon}{\lambda^{2}+\epsilon^{2}} \nonumber\\
   &=& \frac{1}{\pi}G^{\pm}(w,w)\lim_{\epsilon\downarrow 0}\tan^{-1}\left(-\frac{\epsilon}{\lambda}\right),\label{ghde18}
\end{eqnarray}
where $\lim_{\epsilon\downarrow 0}\tan^{-1}\left(-\frac{\epsilon}{\lambda}\right)=\pi \mathbb{H}(\lambda)$ is the Heisenberg distribution (see~\cite{KRP}, \S2.7, Example 4, p.45) and $G^{\pm}(w,w)\equiv G^{\pm}(w)$ is given by
$$G^{\pm}(w)= \sum_{m=1}^{p}\sum_{k=0}^{\sigma_{\frac{n}{2}}-1}\sum_{l=0}^{\tau_{\frac{n}{2}}-1}h_{kl}a^{-l}(-1)^{m-1}
  \epsilon_{0} \delta(w).$$
  Therefore, one gives the spectral shift function in times of Heisenberg distribution as
  $$\xi^{\pm}(\lambda,H_{\frac{n}{2},e},H_{0})=G^{\pm}(w)\mathbb{H}(\lambda).$$
This completes the proof.
\end{proof}
\begin{defn} An operator $T\in B(\mathscr{H})$ is called Hilbert-Schmidt if and only if $tr (T^{\ast}T)<\infty$. The family of Hilbert-Schmidt operators is denoted by $B_{2}(\mathscr{H}).$
\end{defn}
Next, it is noteworthy to consider the Hilbert-Schmidt property of $G^{\pm}(w)$ in $L^{2}(\Omega\times\Omega,\vd \mu_{\omega}(w)\otimes\vd \mu_\omega(w)).$
 $G^{\pm}(w)$ is defined in $L^{2}(\Omega\times\Omega,\vd \mu_{\omega}(w)\otimes\vd \mu_\omega(w)).$
\begin{proof}
Let $c_{n}=a^{\tau_{\frac{n}{2}}-1}(-1)^{\sigma_{\frac{n}{2}}+\tau_{\frac{n}{2}}-2}$ and $\mathscr{K}_{p}$ be given by
\begin{equation*}
  \mathscr{K}_{p} = \sum_{m=1}^{p}\sum_{m=0}^{\sigma_{\frac{n}{2}}-1}\sum_{n=0}^{\tau_{\frac{n}{2}}-1}h_{kl} a^{-l}(-1)^{m-1}\epsilon_{0}.
\end{equation*}
Then
\begin{eqnarray}
  \|G^{\pm}\|_{2}^{2} &=& \int_{\Omega}\int_{\Omega}G^{\pm}(\cdot,w,z)\overline{G^{\pm}(\cdot,w,z)}\vd\mu_{\omega}(w)\vd\mu_{\omega}(w)\nonumber \\
  &=&\int_{\Omega}\int_{\Omega} \mathscr{K}_{p}^{2}\left(\sum_{m=0}^{\sigma_\frac{n}{2}-1}\sum_{n=0}^{\tau_\frac{n}{2}-1}c_{n}h_{mn}w^{m+n+\rho_\frac{n}{2}-1}\right)^2\delta^2(w)\vd w\vd w\nonumber\\
  &=&\left(\int_{\Omega}\mathscr{K}_{p}\sum_{m=0}^{\sigma_\frac{n}{2}-1}\sum_{n=0}^{\tau_\frac{n}{2}-1}c_{n}h_{mn}w^{m+n+\rho_\frac{n}{2}-1}\delta(w)\vd w\right)^{2} .\nonumber\\
  &=&\left(\int_{\Omega}\mathscr{K}_{p}\sum_{m=0}^{\sigma_\frac{n}{2}-1}\sum_{n=0}^{\tau_\frac{n}{2}-1}c_{n}h_{mn}w^{m+\beta}\delta(w)\vd w\right)^{2} .\nonumber\\
  \label{Hsc}
\end{eqnarray}
Let $\displaystyle\omega(w)=\sum_{m=0}^{\sigma_\frac{n}{2}-1}\sum_{n=0}^{\tau_\frac{n}{2}-1}c_{n}h_{mn}w^{m+n+\rho_\frac{n}{2}-1}$ then equation~\eqref{Hsc} becomes
\begin{equation*}
  \|G^{\pm}\|_{2}^{2}=\mathscr{K}_{p}^{2}\omega^{2}(0)<\infty.
 \end{equation*}
  Hence, the result follows.
 \end{proof}
 In what follows, using equation~\eqref{ddff1} as $p\rightarrow\infty, \mathscr{K}_{p}\rightarrow\mathscr{K}_{\infty}, \phi(w)\rightarrow \phi_{\infty}, G_{\infty}^{\pm}(\cdot,w,z)$ takes the form
 \begin{equation*}
   G_{\infty}^{\pm}(\cdot,w,z)=\mathscr{K}_{\infty}\phi_{\infty}(w)\delta(z).
 \end{equation*}
 \begin{teo}\label{GSI} Let  $\mathscr{H}=L^2(\Omega,\vd \mu_{\omega})$  and
  $$T_{G_{\infty}}^{\pm}\Psi(z):=\int_{\Omega}G_{\infty}^{\pm}(\cdot,w,z)\Psi(w)\vd \mu_{\omega}(w)$$
  then $T_{G_{\infty}}^{\pm}\in B(\mathscr{H})$ is Hilbert-Schmidt if and only if there is a function $G_{\infty}^{\pm}\in L^{2}(\Omega\times\Omega,\vd \mu_{\omega}(w)\otimes\vd \mu_\omega(w)).$ Moreover,
  $$\|T_{G_{\infty}}^{\pm}\|_{2}^{2}=\int_{\Omega}|G_{\infty}^{\pm}(\cdot,w,z)|^{2}\vd \mu_{\omega}(w)\vd \mu_{\omega}(w).$$
  \end{teo}
  \begin{proof}
    It has been shown by Theorem~\ref{GSI} that a finite rank operator $G^{\pm}\in L^{2}(\Omega\times\Omega, \vd\mu_{\omega}\otimes\vd\mu_{\omega})$ (say of rank $p$) and $T_{G}^{\pm}$ its associated integral operator. Now, as $p\rightarrow \infty, G^{\pm}\rightarrow G_{\infty}^{\pm}\in L^{2}(\Omega\times\Omega, \vd\mu_{\omega}\otimes\vd\mu_{\omega}), T_{G}^{\pm}\rightarrow T_{G_{\infty}}^{\pm}.$ For any two functions $\Psi_{1},\Psi_{2}\in L^{2}(\Omega,\vd\mu_{\omega}), T_{G_{\infty}}^{\pm}\Psi_{1}=T_{G_{\infty}}^{\pm}\Psi_{2}\implies  T_{G_{\infty}}^{\pm}(\Psi_{1}-\Psi_{2}) = 0\implies \Psi_{1}-\Psi_{2}=0$ as $T_{G_{\infty}}^{\pm}\neq 0.$ This means that $\Psi_{1}=\Psi_{2}$ and thus $T_{G_{\infty}}^{\pm}$ is well defined on $\mathscr{H}.$

    However,
    \begin{eqnarray*}
      \|T_{G_{\infty}}^{\pm}\Psi(z)\|_{L^1}&=&\bigg\|\int_{\Omega}G_{\infty}^{\pm}(\cdot,w,z)\Psi(w)\vd\mu_{\omega}(w)\bigg\|_{L^1} \\
      \implies   \|T_{G_{\infty}}^{\pm}\Psi(z)\|_{L^1}\leqslant\|T_{G_{\infty}}^{\pm}\|_{L^2}\|\Psi(z)\|_{L^2} &\leqslant& \|G_{\infty}^{\pm}(\cdot,w,z)\|_{L^2}\|\Psi(w)\|_{L^2}\\
      \implies \|T_{G_{\infty}}^{\pm}\|_{L^2}&\leqslant& \|G_{\infty}^{\pm}(\cdot,w,z)\|_{L^2}
    \end{eqnarray*}
    (by Cauchy-Schwartz inequality). From the last inequality, it is obvious that as $p\rightarrow\infty$
    $$\|G^{\pm}-G_{\infty}^{\pm}\|_{L^2}\rightarrow 0\implies \|T_{G}^{\pm}- T_{G_{\infty}}^{\pm}\|_{L^2}\rightarrow 0. $$
    Thus, $T_{G_{\infty}}^{\pm}$ is compact and it is however remarkable that if $\{\Psi_{\nu}(z)\overline{\Psi_{\nu'}(w)}\}_{\nu,\nu'\in\mathbb{N}}$ is orthonormal basis of $L^{2}(\Omega\times\Omega,\vd\mu_{\omega}\otimes\vd\mu_{\omega})$  then the trace operator
    \begin{eqnarray*}
      tr (T_{G_{\infty}}^{\pm\,\ast}T_{G_{\infty}}^{\pm}) &=&  \langle T_{G_{\infty}}^{\pm}\Psi_{\nu},T_{G_{\infty}}^{\pm}\Psi_{\nu}\rangle, \\
      &=&\langle T_{G_{\infty}}^{\pm\,\ast}T_{G_{\infty}}^{\pm}\Psi_{\nu},\Psi_{\nu}\rangle\\
       &=& \| T_{G_{\infty}}^{\pm}\|_{L^2}^2\|\Psi_{\nu}\|_{L^2}^2\\
       &=&\| T_{G_{\infty}}^{\pm}\|_{L^2}^2\leqslant \|G_{\infty}^{\pm}(\cdot,w,z)\|_{L^2}<\infty.\\
    \end{eqnarray*}
    when $\| T_{G_{\infty}}^{\pm}\|_{L^2}^2= \|G_{\infty}^{\pm}(\cdot,w,z)\|_{L^2}, T_{G_{\infty}}^{\pm}\in B_{2}(\mathscr{H}).$
  \end{proof}
  It is observed that the kernel $G(z,w)$ is a differentiable function. Hence, the following result.
 \begin{teo} Let $\psi\in\mathscr{C}^{\infty}(\mathbb{CP}^1)$- then the integral operator
  $$T(\psi(z))=\int_{\mathbb{CP}^1}G(z,w)\psi(z)\vd g[z].$$
defined on $L^{2}(\mathbb{CP}^1)=L^2(\mathbb{CP}^1,\vd g[z])$ is a Hilbert-Schmidt operator, hence is compact.
\end{teo}
The trace of $T$ is given by
$$\mathrm{tr}\;(T)=\int_{\mathbb{CP}^1}G(z,z)\vd g[z].$$
Since $\mathbb{CP}^1$ is a compact Riemann manifold, we observe that the closed graph theorem implies that the map
$$\psi\mapsto\mathrm{tr}\;(T(\psi)):\mathscr{C}^{\infty}(\mathbb{CP}^1)\longrightarrow \mathbb{C}$$
is a distribution on $\mathbb{CP}^1.$
This result shows that the Lie algebraic HDO is a compact operator defined on a compact Lie group $SL(2,\mathbb{C})$.
\section{Conclusion}
In this work, the quasi-exact and exact solvability of the Lie algebraic equations associated with the $sl(2)$-algebra has been studied relative to the Heun operator.  The distributional solution of QES Heun equation, the Green kernel of exactly solvable Heun equation, the integral relations associated with Heun- Green kernels obtained, the compactness of the integral operator SSF and Hilbert-Schmidt condition of the integral operator associated with Heun-Green kernel has been discussed. The spectral shift functions of the exactly solvable operator has been investigated using the Fourier transform technique.

 \vspace{.5cm}
\address{Department of Mathematics, \\ Adeyemi Federal University of Education,\\ 143 Ondo-Ore Road, 351103 Ondo State, NIGERIA}

\noindent{\emph{E-mail address}:~\email{idiongus@aceondo.edu.ng, usidiong@gmail.com}}

 \end{document}